\documentclass[11pt]{amsart}

\usepackage{a4wide}
\usepackage{amsmath,amssymb} 
\usepackage{bbm}

\theoremstyle{plain}
\newtheorem{theorem}{Theorem}
\newtheorem{prop}[theorem]{Proposition}
\newtheorem{lemma}[theorem]{Lemma}

\newtheorem{fact}[theorem]{Fact}

\theoremstyle{definition}
\newtheorem{definition}{Definition}
\newtheorem{remark}{Remark}

\newcommand{\Z}{{\mathbb Z}}
\newcommand{\Q}{{\mathbb Q}}
\newcommand{\R}{{\mathbb R}}
\newcommand{\N}{{\mathbb N}}
\newcommand{\C}{{\mathbb C}}

\newcommand{\A}{{\mathcal A}}
\newcommand{\la}{{\mathcal L}}

\newcommand{\im}{{\mathrm{i}}}

\newcommand{\e}{\operatorname{e}}
\newcommand{\F}{\mathcal F}

\newcommand{\oplam}{\mbox{\Large $\curlywedge$}}

\begin{document}

\title{Diffraction intensities of a class of binary\\[2mm]
  Pisot substitutions via exponential sums}

\author{Timo Spindeler}
\address{Fakult\"at f\"ur Mathematik, Universit\"at Bielefeld, \newline
\hspace*{\parindent}Postfach 100131, 33501 Bielefeld, Germany}
\email{tspindel@math.uni-bielefeld.de}

\begin{abstract}  
This paper is concerned with the study of diffraction intensities of a 
relevant class of binary Pisot substitutions via exponential sums. 
Arithmetic properties of algebraic integers are used to give a new 
and constructive proof of the fact that there are no diffraction 
intensities outside the Fourier module of the underlying cut and 
project schemes. The results are then applied in the context of random
substitutions.
\end{abstract}



\maketitle

\section{Introduction}

The aim of mathematical diffraction theory is to describe the 
structure of point configurations in space (which model crystals 
and quasicrystals) through the associated autocorrelation and 
diffraction measures. Bombieri and Taylor were among the first
to raise the question which distributions of matter diffract, 
i.e. show sharp spots (or Bragg peaks) in their diffraction patterns; 
see \cite{bt}. There are two successful approaches to generate 
such structures. The first is by creating certain tilings by the 
method of inflation followed by decomposition using a finite set 
of proto-tiles. The second is by creating point sets through the 
method of cut and project sets, which are also called model sets. 
Either way, it is desirable to obtain explicit formulas for the 
associated diffraction measure; see \cite[Thm.~9.4]{bg} and 
\cite[Prop.~9.9]{bg}. However, some parts of the proof of 
\cite[Prop.~9.9]{bg} are not constructive and require knowledge 
of abstract harmonic analysis, which is due to \cite[Prop.~4.5.1]{mey}.

The objective of this paper is to give a constructive and elementary 
proof of pure point diffraction for a certain class of binary Pisot 
substitutions via exponential sums. This will be done in Section 3, 
Theorem~$\ref{thm:outside}$. The key ingredients are arithmetic
properties of powers of algebraic integers $\alpha$, i.e. expressions
of the form $\{\xi\alpha^n\}$, $\xi\in\R$; compare 
Lemmas~$\ref{lem:help-1}$ and ~$\ref{lem:help-2}$. 
Here, $\{x\}$ denotes the fractional part of $x$, $[x]$ denotes the integer
part of $x$ and 
\[
\|x\|:= \min \{ |x-m| \, |\, m\in \Z\}.
\]
These were investigated by Dubickas in \cite{du2,du1}. In Section 4, 
the main result will be extended to random inflation tilings, i.e. the
result states that there are no pure point diffraction intensities outside 
the Fourier module for a special class of stochastic substitutions. These
one-dimensional tilings, which belong to a class of tilings first considered 
by Godr\`eche and Luck in \cite{gl1}, extend the study of conventional 
substitutions and introduce the notion of local mixtures of substitution 
rules on the basis of a fixed probability vector; see \cite{mo,mo2} for 
further details.

\section{Preliminaries} 

The purpose of this section is to summarise results from mathmatical 
diffraction theory; see \cite[Chs. 8 and 9]{bg} for general background. 
Let $P$ be an infinite uniformly discrete point set in $\R$. Define the 
attributed \emph{Dirac comb} by
\[
   \delta_P \, := \sum_{x\in P} \delta_x 
    \quad \text{ together with } \quad 
   \widetilde{\delta_P} \, := \sum_{y\in P} \delta_{-y}
\]
and study the properties of the family of measures 
$\{\gamma_P^{(n)}\, |\, n>0\}$ with
\[
   \gamma_P^{(n)}\, := \, \gamma_{\delta_P}^{(n)} \, := \, 
    \frac{\delta_{P_n}* \widetilde{\delta_{P_n}}}
     {\operatorname{vol}(B_n)} 
\]
and $P_n:=P\cap B_n(0)$. 

It is not clear that the sequence $(\gamma_P^{(n)})_{n\in\N}$
converges. Each $\gamma_P^{(n)}$ is well-defined (since $\delta_{P_n}$
is a finite measure with compact support) and positive definite by
construction. Every accumulation point of $\{\gamma_P^{(n)}\ |\ n>0\}$
in the vague topology is called an \emph{autocorrelation measure}
of $\delta_P$, and as such it is a positive definite measure by
construction. If only one accumulation point exists, the
autocorrelation measure
\[
     \gamma_P:=\lim_{n\to\infty} \gamma_P^{(n)}
\]
is well-defined and Fourier transformable. Its Fourier transform 
$\widehat{\gamma_P}$ is called \emph{diffraction measure}.
In the case of a cut and project scheme (CPS), there are explicit 
formulas for the diffraction measure;
for a detailed introduction of cut and project schemes,
we refer the reader to \cite[Ch. 7]{bg}, \cite{moo} as well as
\cite{mey2}.

\begin{theorem}\cite[Thm.~9.4]{bg}\label{thm:diffract-1}
  Let $\Lambda=\oplam(W)$ be a regular model set for the CPS
  $(\R,H,\la)$ with compact window $W=\overline{W^{\circ}}$ and
  autocorrelation $\gamma_{\Lambda}$. The diffraction measure
  $\widehat{\gamma_{\Lambda}}$ is a positive and positive definite,
  translation bounded measure. It is explicitly given by
\[
     \widehat{\gamma_{\Lambda}} \, =
   \sum_{k\in L^{\circledast}}I(k)\delta_k \, ,   
\]
where the diffraction intensities are $I(k)=|A(k)|^2$ with
the amplitudes
\[
    A(k) \, = \, 
   \frac{\operatorname{dens}(\Lambda)}
    {\mu_H(W)}\widehat{1_W}(-k^{\star})
\] 
and supporting set $L^{\circledast} = \pi_1(\la^*)$, and $\la^*$
is the dual lattice of $\la$.  \qed
\end{theorem}

Furthermore, there is an alternative approach via exponential sums 
which is justified by  \cite[Thm.~3.4]{hof}, because we deal with
pure point measures.

\begin{prop}\cite[Prop.~9.9]{bg}\label{prop:diffract-1}  
  Consider a regular model set $\Lambda=\oplam(W)$ for the 
  CPS $(\R,\R,\la)$, with compact window
  $W=\overline{W^{\circ}}$ and Fourier module $
  L^{\circledast}=\pi_1(\la^*)\subseteq \R$. Then, one has
\[
    \frac{1}{|B_N|} \sum_{x\in\Lambda_N} e^{-2\pi\im kx} \
    \xrightarrow{N\to\infty}\ \begin{cases}
    A(k), & k\in  L^{\circledast}, \\
     0, & \text{otherwise}, \end{cases}
\]
where $A(k)$ is the amplitude of Theorem~$\ref{thm:diffract-1}$ for
the internal space $H=\R$.   \qed 
\end{prop}

The proof of this result is constructive when $k\in L^{\circledast}$
or $k\in \Q( L^{\circledast})$, in the sense that it uses an explicit
convergence argument via exponential sums. For $k\notin \Q(
L^{\circledast})$, however, it uses an abstract argument from 
\cite{mey}. In what follows, we demonstrate an alternative approach 
via exponential sums, for a relevant class of Pisot substitutions.

\section{Intensities of a certain class of Pisot substitutions}

Consider the binary alphabet $\A=\{a,b\}$ with the dictionary $\A_2^*$
 and the substitution 
\[
   \sigma:\A_2^* \to \A_2^*, \quad 
   \sigma: \begin{cases} 
     a\mapsto w(a,b) \\
     b\mapsto a\end{cases}, 
\]
where $w(a,b)$ is a word in $a$ and $b$, in which the letter $a$
occures $p$ times and the letter $b$ occures $q$ times with
$p,q\in\N$, $p\ge q$.  The eigenvalues of the substitution matrix
\[
   M_{\sigma}= \begin{pmatrix} 
     p & 1 \\ q & 0 \end{pmatrix}
\]
are
\[ 
   \vartheta\, := \, \frac{p+\sqrt{p^2+4q}}{2} 
    \quad \text{ und } \quad 
   \vartheta' \, := \, \frac{p-\sqrt{p^2+4q}}{2} .
\]
The assumption $q\leq p$ guaranties that $\sigma$ is a \emph{Pisot
  substitution}. A primitive substitution is called Pisot
substitution if the Perron--Frobenius eigenvalue
$\lambda_{\textrm{PF}}$ is a Pisot--Vijayaraghavan (PV) number, i.e.\
an algebraic integer strictly greater than $1$ whose conjugates lie
inside the open unit disc, $\{z\in\C\ |\ |z|<1\}$.

The left eigenvector $v=(\vartheta,1)$ gives rise to a
\emph{geometric} realisation of $\sigma$ as a tiling of the real
line by two types of intervals; see \cite[p.~74]{bg}.  To calculate the
diffraction intensities of these tilings, we need some preparation.

Define $w^{(0)}:=b$ and $w^{(n)}:=\sigma^n(b)$ for $n\in\N$. This
sequence of words $(w^{(n)})_{n\in\N}$ satisfies a \emph{concatenation
  rule}, i.e. if $w(a,b)=w_0\ldots w_{|w(a,b)|-1}$, we have
\[ 
w^{(n)}=w^{(j_0)}\cdots w^{(j_{|w(a,b)|-1})},
\]
where $j_i=n-1$ if $w_{i}=a$ and $j_i=n-2$ if $w_{i}=b$ 
(in the case $w(a,b)=ab$, the Fibonacci chain, this is
$w^{(n)}=w^{(n-1)}w^{(n-2)}$). Moreover, we consider a linear
recursion
\begin{equation} \label{a}
   \F_n \, = \, p\F_{n-1}+q\F_{n-2}, \quad \text{with} \quad
   \F_0=0 \text{ and } \F_1=1,  
\end{equation}
which has the unique solution
\[
    \F_n \, = \, \frac{1}{\sqrt{p^2+4q}}
   \left(\vartheta^n-\vartheta'^n\right).
\]
Via induction, we get

\begin{fact} \label{fact}
  For all $n\in\N_0$, one has the identity
 $\,\vartheta^{n+2} = \F^{}_{\! n+2}\vartheta+q\F^{}_{\! n+1} $. 
\end{fact}

\begin{remark}
  In the case $p=q=1$, these are the well-known relations between the
  Fibonacci numbers $F_n$ and the golden ratio $\tau$, since
\[
    F_n \, = \, \frac{1}{\sqrt{5}}\big(\tau^n-\tau'^n\big) 
    \quad \text{ and }\quad 
     \tau^{n+2}=F_{n+2}\tau+ F_{n+1}.
\]
On the other hand, the case $n=0$ gives us $\vartheta^2=p\vartheta+q$.
Note that we also have $\ell(w^{(n)})=\vartheta^n$.
\end{remark}

In what follows, the Fourier amplitude for the geometric patch defined by
$w^{(n)}$ is denoted by $A_n(k)$, i.e.
\[
    A_n(k)=\sum_{j=1}^{|w^{(n)}|} \e^{-2\pi \text{i}kx_j}
\]
(note that this quantity is not normalised per point), where $|w^{(n)}|$
is the number of letters of $w^{(n)}$. Following \cite[Sec.~2]{gl1} and 
using the concatenation rule, one can deduce
\[
    A_n(k) \, = \, 
    f_{n-1}(k)A_{n-1}(k)+g_{n-2}(k)A_{n-2}(k).
\]
Here, $g_n$ is a sum of $q$ exponential functions of the form
$\e^{-2\pi \im k \phi_0}$  and $f_n$ is of the form $1+
\e^{-2\pi \im k\vartheta^{n_0}}+\sum_{j=1}^{p-2}\e^{-2\pi\im k\phi_j}$, 
where $n_0\in\{n-1,n-2\}$ and $\phi_j\in\R$ for all $j\in\{0,\ldots,p-2\}$. 
 By \cite[Thm.~3.2]{hof}, the intensities can be calculated as
\begin{equation}\label{b}
   I(k) \, = \lim_{n\to\infty}
  \frac{|A_n(k)|^2}{\vartheta^{2n}}. 
\end{equation} 
Any of the Pisot substitutions under investigation here, due to
Hollander and Solomyak \cite{hs} and Sing \cite{sin}, possesses a
description as a model set, so that the constructive part of
Proposition~\ref{prop:diffract-1} applies to all $k\in \Q(\vartheta)$. In 
particular, one then has
\begin{equation} \label{1}
    \lim_{n\to\infty} \frac{A_n(k)}{\vartheta^n}=A(k)
    \quad \text{ and } \quad I(k) \, = \, |A(k)|^2
\end{equation}
with $A(k)$ according to Theorem~\ref{thm:diffract-1}, where the
construction of the underlying CPS follows from \cite{bm,sin}. 

To determine the limit in Eq. ~(\ref{b}) for all remaining $k\in \R$, we need the
following three Lemmata.

\begin{lemma}\cite[Cor.~2]{du1}\label{lem:help-1} 
  If\/ $\alpha$ is a PV number and\/ $\xi\notin\Q(\alpha)$, one has
\[
    \limsup_{n\to\infty} \{\xi\alpha^n\} - 
   \liminf_{n\to\infty}\{\xi\alpha^n\}
   \, \ge \, \frac{1}{1+\alpha}.  
   \hfill  
\]  \qed 
\end{lemma}

\begin{lemma}\cite[Thm.~1]{du2}\label{lem:help-2}
  Let\/ $\alpha>1$ be an algebraic integer and let\/ $\xi>0$ be a real
  number. Then, the set\/ $\big\{\{\xi\alpha^n\}\ |\ n\in\N \big\}$
  has only finitely many limit points if and only if\/ $\alpha$ is a
  PV number and\/ $\xi\in\Q(\alpha)$.  \qed
\end{lemma}

\begin{lemma}\label{lem:help-3}
  Let $\alpha$ be a PV number, $\xi\notin\Q(\alpha)$ and
  $(y_n)_{n\in\N}:=(\xi\alpha^n)_{n\in\N}$. Then, there are
  numbers $\delta=\delta(\xi,\alpha)\in(0,1)$ and
  $r=r(\xi,\alpha)\in\N$ such that for all $n\in\N$:
\[
   \|y_n\|   < \delta \, \implies \,
   \|y_j \| \ge \delta \text{ for at least one }
    j\in\{n+1,\ldots,n+r\}.
\]
\end{lemma}
\begin{proof} Suppose the assertion is wrong. Then, for any
  $\delta\in(0,1)$, there is a strictly increasing sequence
  $(r_m)_{m\in\N}$ of positive integers (hence also unbounded) and a
  subsequence $(y_{n_m})_{m\in\N}$ of $(y_n)_{n\in\N}$, such that
\[
     \|y_{n_m}\| < \delta \, \implies \,
     \|y_j\| < \delta \text{ for all } 
     j\in\{n_m+1,\ldots,n_m+r_m\}.
\]
This cannot be true for any $\delta$ by Lemma~\ref{lem:help-2},
because $(\{y_n\})_{n\in\N}$ has infinitely many limit points and the distance
between the biggest and smallest limit point is at least
$\frac{1}{1+\alpha}$ by Lemma~\ref{lem:help-1}.
\end{proof}

Now, we can determine the intensities for wave numbers
$k\notin\Q(\vartheta)$ as follows.

\begin{theorem}\label{thm:outside}
  For any $k\in \R\setminus \Q(\vartheta)$, one has $I(k)=0$.  
\end{theorem}
\begin{proof}
  Obviously, it is sufficient to show via induction that
\[
    \forall k\notin\Q(\vartheta):\ 
    \exists c=c(k)>0,n_0=n_0(k)\in\N: \ 
     \forall n\ge n_0:\quad 
    \frac{|A_n(k)|^2}{\vartheta^{2n}}\leq \frac{c}{n}.
\]
For $n\in\{n_0,\ldots,n_0+2r\}$ ($r$ will be chosen later), one just
has to choose the constant $c$ large enough. Now, let the assertion be
true for a fixed $n= n_0+2r$ and its predecessors $n_0,\ldots,
n_0+2r-1$. Without loss of generality, let $n_0$ be so large that
\[
    \frac{n_0+1}{n_0-r-1} \, \leq \, 1+\varepsilon
\]
($\varepsilon>0$ will be chosen later, too). By Lemma~\ref{lem:help-3}, 
there is a $\delta\in(0,1)$ and $r\in\N$, such that $\|k\vartheta^j\|\ge \delta$ 
for at least one $j\in\{n-r,\ldots,n\}$, i.e. 
\[
    |1+\e^{-2\pi\im k\vartheta^j}| \leq 2- \delta'
\]
for some $\delta'\in(0,1)$. Thus, we have 
\[
|f_j| \leq p-\delta' 
\]
for at least one $j\in\{n-r,\ldots,n\}$ each. Without loss of 
generality, we can assume that $j=n-r$. It follows by Eq. ~(\ref{a})
and $|f_n|\leq p$ and $|g_n|\leq q$
\begin{equation*} 
\begin{split}
   |A_{n+1}| \, &= \, |f_nA_n+g_{n-1}A_{n-1}| \leq p|A_n|+q|A_{n-1}|                              \\
    &=\,\F_2|A_n|+q\F_1|A_{n-1}|                                                                                         \\
    &=\, \F_2|f_{n-1}A_{n-1}+g_{n-2}A_{n-2}|+q\F_1|A_{n-1}|                                          \\
    &\leq\, \F_2\left(p|A_{n-1}|+q|A_{n-2}|\right)+q\F_1|A_{n-1}|                                        \\
    &=\, (p\F_2+q\F_1)|A_{n-1}|+q\F_2|A_{n-2}|                                                                   \\
    &=\, \F_3|A_{n-1}|+q\F_2|A_{n-2}|                                                                                  \\
    &\leq\, \ldots                                                                                                                      \\
    &\leq\, \F_{r+1}|A_{n-r+1}|+q\F_{r}|A_{n-r}|                                                                 \\
    &=\, \F_{r+1}|f_{n-r}A_{n-r}+g_{n-r-1}A_{n-r-1}|+q\F_{r}|A_{n-r}|                             \\
    &\leq\, \F_{r+1}\big(|f_{n-r}||A_{n-r}|+|g_{n-r-1}||A_{n-r-1}|\big)+q\F_{r}|A_{n-r}|     \\
    &\leq\, \F_{r+1}\big((p-\delta')|A_{n-r}|+q|A_{n-r-1}|\big)+q\F_{r}|A_{n-r}|    \\
    &=\, \big((p-\delta')\F_{r+1}+q\F_r\big)|A_{n-r}| + q\F_{r+1}|A_{n-r-1}|         \\
    &=\, (\F_{r+2}-\delta'')|A_{n-r}| + q\F_{r+1}|A_{n-r-1}|
\end{split} 
\end{equation*}
for $\delta'':=\F_{r+1}\delta'>0$. Therefore, there is a $\delta''>0$  and $r\in\N$ such that
\begin{equation}\label{c}
    |A_{n+1}| \, \leq \,
     (\F_{r+2}-\delta'')|A_{n-r}|+
      q\F_{r+1}|A_{n-r-1}|. 
\end{equation}
By the induction hypothesis and Eq. ~(\ref{c}), we have
\begin{equation*} 
\begin{split} 
  \frac{|A_{n+1}|^2}{\vartheta^{2n+2}}\,
  &\leq\, \left(\frac{(\F_{r+2}-\delta'')|A_{n-r}|+
    q\F_{r+1}|A_{n-r-1}|}{\vartheta^{n+1}}\right)^2 \\
  &\leq\, \left(\frac{\F_{r+2}-\delta''}{\vartheta^{r+1}}\cdot 
    \sqrt{\frac{c}{n-r}}+\frac{q\F_{r+1}}
     {\vartheta^{r+2}}\cdot \sqrt{\frac{c}{n-r-1}} \right)^2. 
\end{split} 
\end{equation*}
The right hand side is bounded by $\frac{c}{n+1}$ if and only if
\[
   S \,:= \,\left(\frac{\F_{r+2}-\delta''}{\vartheta^{r+1}}\cdot 
   \sqrt{\frac{n+1}{n-r}}+\frac{q\F_{r+1}}
     {\vartheta^{r+2}}\cdot \sqrt{\frac{n+1}{n-r-1}} 
    \right)^2 \leq \, 1.
\]
This in turn is true because
\begin{equation*} 
\begin{split}  
   S \, &\leq \, \vartheta^{-2r-2} \cdot \left( 
         (\F_{r+2}-\delta'')\cdot \sqrt{1+\varepsilon}+ 
         \frac{q\F_{r+1}}{\vartheta}\cdot 
         \sqrt{1+\varepsilon}\right)^2 \\
     &= \, \vartheta^{-2r-2} \cdot \left( (\F_{r+2}-\delta'')+ 
        \frac{q\F_{r+1}}{\vartheta}\right)^2
        \cdot(1+\varepsilon) \, \leq \, 1
\end{split} 
\end{equation*}
holds if and only if
\[
   \varepsilon \, \leq \, \vartheta^{2r+2} \cdot 
    \left( (\F_{r+2}-\delta'')+ \frac{q\F_{r+1}}
       {\vartheta}\right)^{-2} -1.
\]
Now, by Fact~\ref{fact}, $\varepsilon$ can be chosen positive because
\begin{equation*} 
\begin{split}
  0 \, &= \, \vartheta^{2r+2} \cdot(\vartheta^{r+1})^{-2} -1 \\
       &=\,  \vartheta^{2r+2} \cdot \left( \F_{r+2}+ 
             \frac{q\F_{r+1}}{\vartheta}\right)^{-2} -1 \\
       &<\, \vartheta^{2r+2} \cdot \left( (\F_{r+2}-\delta'')
       + \frac{q\F_{r+1}}{\vartheta}\right)^{-2} -1.
\end{split} 
\end{equation*}
By induction, the assertion is true for all $n\ge n_0$.
\end{proof}

\medskip

Let us comment on some connections with known results in the literature.
There is the following link to substitution dynamical 
systems. Let $\zeta$ be a primitive Pisot substitution on $\A=\{a,b\}$ (for 
example one of the substitutions considered above) and 
$w$ be a fixed point of $\zeta$, i.e. an element of $\A^{\Z}$ such that 
$\zeta(w)=w$. Let $S:\A^{\Z}\to \A^{\Z}$ be the shift map defined by
\[
(Sv)_k:=v_{k+1},
\]
and let 
\[
X_{\zeta}:=
\overline{\{S^jw\ |\ j\in\Z\}}\subseteq\A^{\Z}. 
\]
The pair $(X_{\zeta},S)$
is a topological dynamical system, called the substitution dynamical system 
(for $\zeta$). In this situation, it is a well-known fact that $(X_{\zeta},S)$ is
uniquely ergodic, i.e. there is a unique $S$-invariant Borel probability measure 
$\mu$. The system $(X_{\zeta},S,\mu)$ is a measure-preserving system 
and its spectral type is, by definition, the spectral type of the unitary operator 
$U$ on $L^2(X_{\zeta},\mu)$ defined by
\[
Uf(x) = f(Sx).
\]
Furthermore, due to \cite[Thm. 2.2]{hs}, we know that $(X_{\zeta},S,\mu)$ is
pure point (or has
pure discrete spectrum). This means that there is a basis of $L^2(X_{\zeta},\mu)$
consisting of eigenfunctions of $U$. By the
Halmos--von Neumann Theorem, a  measure-preserving transformation is
pure point if and only if it is measure-theoretically isomorphic
to a translation on a compact  Abelian group, see \cite{wal}. \\
The same holds
true if, instead of considering the substitution dynamical system with $\Z$-action,
we have a look at the corresponding tiling dynamical system with $\R$-action;
see \cite{ss} for definitions and results. This also follows from \cite[Thm. 3.1]{cs}. \\
Now, the connection between these results and diffraction theory is the following
statement. Given a dynamical system on the translation bounded measures 
$(\Omega,\alpha)$ with invariant probability measure $m$, associated unitary 
representation $T_m$ by translation operators and associated diffraction measure 
$\widehat{\gamma_m}$, then $\widehat{\gamma_m}$ is pure point if and only 
if $T_m$ is pure point, see \cite[Thm. 7]{bl}.

The above derivation re-establishes the key result via explicit estimates of the 
underlying exponential sums, thus interpreting $A(k)$ from~(\ref{1}) as an 
amplitude - despite the fact that $\delta_{\Lambda}$ for the corresponding model
set is not Fourier transformable as a measure. It is expected that this phenomenon 
is much more general, though it is presently not clear how to extend the concrete 
approach accordingly; see \cite{bl2}.

\medskip

Let us now turn to consequences of Theorem~\ref{thm:diffract-1} 
outside the realm of deterministic inflation rules.

\section{Outlook}
One can extend the result of Theorem~\ref{thm:outside} as follows. In
\cite{gl1}, Godr\`eche and Luck introduced the concept of \emph{random
inflation tilings}, which extends the study of conventional substitutions. 
A mathematical rigorous treatment of a special class of such substitutions 
can be found in \cite{bmo,mo,mo2}.

\begin{definition}
  A substitution $\rho:\A_n^*\to\A_n^*$ is called
  \emph{stochastic} or  \emph{random} if there
  are $k_1,\ldots,k_n\in\N$ and probability vectors
\[
   \big\{{\boldsymbol p}_i=(p_{i1},\ldots,p_{ik_i})\ |\
   {\boldsymbol p}_i\in[0,1]^{k_i}\  \;\text{ and }\; 
   \sum_{j=1}^{k_i}p_{ij}=1,\ 1\leq i\leq n\big\},
\]
such that
\[
   \rho: \; a_i\mapsto \begin{cases} 
   w^{(i,1)}, & \text{with probability } p_{i1}, \\
    \quad \vdots  & \quad \quad \quad \quad \ \vdots \\
   w^{(i,k_i)}, & \text{with probability } p_{ik_i},
   \end{cases} 
\]
for $1\leq i\leq n$ where each $w^{(i,j)}\in\A_n^*$. The corresponding
stochastic substitution matrix is defined by
\[
    M_{\rho}\, := \big(\sum_{q=1}^{k_j} p_{jq} 
   \operatorname{card}_{a_i}w^{(j,q)} 
   \big)_{1\leq i,j\leq n}\in \operatorname{Mat}(n,\R_{\ge0}).
\]
\end{definition}

\begin{remark}
  As in the deterministic case, a random substitution $\rho$ is called
  primitive if and only if $M_{\rho}$ is a primitive matrix. Note, however, 
that the meaning is now a stochastic one.
\end{remark}

Now, let $m\in\N$ and ${\boldsymbol p}_m=(p_0,\ldots,p_m)$ be a
probability vector, both assumed to be fixed. The \emph{random
  substitution} $\zeta_m:\A_2^*\to \A_2^*$ is defined by
\[
    \zeta_m: \; \begin{cases}
      a\mapsto \begin{cases}
      ba^m, & \text{with probability } p_0,\\
      aba^{m-1}, &  \text{with probability } p_1, \\
      \quad \vdots & \quad \quad \quad \quad \ \vdots \\
      a^{m-1}ba, & \text{with probability } p_{m-1} \\
      a^mb, & \text{with probability } p_m,
       \end{cases} \\
      b\mapsto a,
      \end{cases} 
\]
and the one-parameter family $\mathcal R=\{\zeta_m\}_{m\in\N}$ is
called the family of $\textit{random noble means}$
$\textit{substitutions}$ (RNMS). The stochastic substitution matrix 
is given by
\[
   M_m \, := \, M_{\zeta_m} \, = \, \begin{pmatrix}
      m & 1 \\ 1 & 0  \end{pmatrix},
\]
which is independent of the probability vector ${\boldsymbol p}_m$.
The eigenvalues are $\lambda_m:=\frac{m+\sqrt{m^2+4}}{2}$ and
$\lambda_m':=\frac{m-\sqrt{m^2+4}}{2}$, while the left eigenvector is
$(\lambda_m,1)$. Now, one can prove (almost) along the same lines as
in Theorem~\ref{thm:outside} that $I(k)=0$ for all $k\notin\Q(\lambda_m)$. 
Even more, with a modification of Lemma~\ref{lem:help-3}, one can 
prove the following result.

\begin{prop}
  Let $m\in\N$ and consider the RNMS $\zeta_m$. For any 
wave number $k\in \R\setminus\frac{\Z[\lambda_m]}{\sqrt{m^2+4}}$ 
  (i.e. for any $k$ that is not in the Fourier module), we have $I(k)=0$.
\end{prop}

\begin{proof}[Sketch of proof]
  For $k\notin\Q(\lambda_m)$, one can argue as in the proof of 
  Theorem~\ref{thm:outside}.
  The functions $f_n$ and $g_n$ are again sums of exponential
  functions as above (maybe multiplied by some $p_i$) and are again
  bounded by $m$ respectively $1$.

  For $k\in\Q(\lambda_m)\setminus\frac{\Z[\lambda_m]}{\sqrt{m^2+4}}$,
  the assertion of Lemma~\ref{lem:help-3} is obviously still true and
  one can argue again as in the proof of Theorem~\ref{thm:outside}.
\end{proof}

\begin{remark} 
In a deterministic setting, substitution dynamical systems are rather 
well understood; see e.g. \cite{bg,qu}. Far less is known in the realm of
systems inducing mixed spectra. In this case, the understanding in
the presence of entropy is only at its beginning, and it is desirable to
work out particular examples like the RNMS. For more information 
about the RNMS, see \cite{bmo,mo,mo2}. 
\end{remark}

As before, various generalisations should be possible, in particular in view 
of the fact that Godr\`eche and Luck \cite{gl1} also treat planar analogues. 
At present, it is not clear though how the above approach can be extended 
to cover planar systems.

\section*{Acknowledgments}

The author wishes to thank Michael Baake for helpful 
discussions and two anonymous referees
for useful comments. This work is supported by the German 
Research Foundation (DFG) via the Collaborative Research 
Centre (CRC 701) through the faculty of Mathematics, 
Bielefeld University.


\begin{thebibliography}{xxxxxx}
\bibitem{bg} 
  M.~Baake and U.~Grimm,
  \textit{Aperiodic Order.\ Vol.\ 1:\ A Mathematical Invitation},
  Cambridge University Press, Cambridge (2013).
\bibitem{bl}  
  M.~Baake and D.~Lenz,
  Dynamical systems on translation bounded measures: Pure point dynamical and diffraction spectra,
  \textit{Ergodic Th. \& Dynam. Syst.} \textbf{24} (2004) 1867-1893; 
  \texttt{arXiv:math.DS/0302231}.
\bibitem{bl2}
   M.~Baake and D.~Lenz,
   Spectral notions of aperiodic order,
   \texttt{ arXiv:math.DS/1601.06629v1}.  
\bibitem{bmo} 
  M.~Baake and M.~Moll,
  Random noble means substitutions, in: S.~Schmid, R.L. Withers and R.~Lifshitz (eds.) 
  \textit{Aperiodic Crystals}, Springer, Dordrecht (2013) 19--27; 
  \texttt{arXiv:math.DS/1210.3462}.
\bibitem{bm} 
  M.~Baake and R.V.~Moody,
  Weighted Dirac combs with pure point diffraction,
  \textit{J.\ reine angew.\ Math.\ (Crelle)} \textbf{573} 
   (2004), 61--94; \texttt{arXiv:math.MG/0203030}.
\bibitem{bt}
   E.~Bombieri and J.E.~Taylor,
   Which distributions of matter diffract? An initial investigation,
    \textit{ J. Physique Colloque} \textbf{47} (C3) (1986), 19--28.
\bibitem{cs}    
    A.~Clark and L.~Sadun,
    When size matters: subshifts and their related tiling spaces,
    \textit{Ergodic Th. \& Dynam. Syst.} \textbf{23} (2003), 1043--1058;
    \texttt{arXiv:math/0306214v1}.
\bibitem{du2} 
  A.~Dubickas,
  There are infinitely many limit points of the fractional parts 
  of powers, 
  \textit{Proc.\ Indian Acad.\ Sci.\ (Math.\ Sci.)}
  \textbf{115} (2005), 391-397; \texttt{arXiv:math.NT/0512314}.
\bibitem{du1} 
  A.~Dubickas,
  Arithmetical properties of powers of algebraic numbers,
   \textit{Bull.\ London Math.\ Soc.} \textbf{38} (2006), 70--80.
\bibitem{gl1} 
  C.~Godr\`eche and J.M.~Luck,
  Quasyperiodicity and randomness in tilings of the plane,
  \textit{J.\ Stat.\ Phys.} \textbf{55} (1989), 1--28.
\bibitem{hof} 
  A.~Hof, 
  On diffraction by aperiodic structures,
  \textit{Commun.\ Math.\ Phys.} \textbf{169} (1995), 25--43.
\bibitem{hs} 
  M.~Hollander and B.~Solomyak,
  Two-symbol Pisot substitutions have pure discrete spectrum,
  \textit{Ergod.\ Th.\ \& Dynam.\ Syst.} \textbf{23} (2003) 533--540.
\bibitem{mey}
  I.~Meyer,
  \textit{Nombres de Pisot, Nombres de Salem et Analyse Harmonique},
   LNM 117 Springer, Berlin (1970)
\bibitem{mey2}
  I.~Meyer,   
  \textit{Algebraic Numbers and Harmonic Analysis},
  North Holland, Amsterdam (1972). 
\bibitem{mo} 
  M.~Moll,
  \textit{On a Family of Random Noble Means Substitutions},
  PhD thesis, Univ.\ Bielefeld (2013). 
\bibitem{mo2} 
  M.~Moll,
  Diffraction of random noble means words,
   \textit{J.\ Stat.\ Phys.} \textbf{156} (2014), 1221--1236;
   \texttt{arXiv:1404.7411}.
\bibitem{moo}
   R.V.~Moody, 
   Model sets: A survey, in: F.~Axel, F. D\'{e}noyer and J.P.~Gazeau (eds.) \textit{From Quasicrystals to More Complex Systems}, Springer, Berlin and EDP Sciences, Les Ulis (2000), 145--166;
   \texttt{arXiv:math.MG/0002020v1}.
\bibitem{qu}
   M.~Queff\'{e}lec,
    \textit{Substitution Dynamical Systems - Spectral Analysis},
    2nd. ed. LNM 1294 Springer, Berlin (2010).
\bibitem{sin} 
  B.~Sing,
  \textit{Pisot Substitutions and Beyond},
  PhD thesis, Univ.\ Bielefeld (2006).
\bibitem{ss}  
  V.F.~Sirvent and B.~Solomyak,
  Pure discrete spectrum for one-dimensional substitution systems of pisot type,
  \textit{Canad. Math. Bull. Bol.} \textbf{45} (4) (2002), 697--710.
\bibitem{wal}  
  P.~Walters,
  \textit{An Introduction to Ergodic Theory}, 
  Springer Graduates Texts in Math., Springer-Verlag, New York (1982). 
\end{thebibliography}
\end{document}